\documentclass{amsart}

\usepackage{latexsym,amsmath,amssymb,amsfonts,amscd,graphics}
\usepackage{blindtext}
\usepackage[left=3cm,right=3cm]{geometry}
\usepackage[mathscr]{eucal}
\usepackage{mathrsfs}
\usepackage{tikz-cd}
\usepackage{hyperref, cleveref}
\usepackage{thmtools}
\hypersetup{
	colorlinks=true,
	linkcolor=blue,
	filecolor=blue,      
	urlcolor=cyan,
	citecolor=magenta}

\numberwithin{equation}{section}
\theoremstyle{plain}
\newtheorem{theorem}[equation]{Theorem}

\newtheorem{lemma}[equation]{Lemma}
\newtheorem{proposition}[equation]{Proposition}
\newtheorem{corollary}[equation]{Corollary}

\theoremstyle{remark}
\newtheorem{remark}[equation]{Remark}

\theoremstyle{definition}
\newtheorem{definition}[equation]{Definition}

\newtheorem{example}[equation]{Example}

\newcommand{\bP}{\mathbb{P}}
\newcommand{\bA}{\mathbb{A}}

\newcommand{\bQ}{\mathbb{Q}}
\newcommand{\bZ}{\mathbb{Z}}
\newcommand{\bF}{\mathbb{F}}
\newcommand{\bC}{\mathbb{C}}

\newcommand{\bL}{\mathbb{L}}

\newcommand{\calC}{\mathcal{C}}

\newcommand{\calT}{\mathcal{T}}

\newcommand{\calO}{\mathcal{O}}

\newcommand{\calI}{\mathcal{I}}
\newcommand{\calX}{\mathcal{X}}

\newcommand{\calJ}{\mathcal{J}}

\newcommand{\Sym}{\mathrm{Sym}}

\newcommand{\Proj}{\mathrm{Proj}}

\newcommand{\Sing}{\mathrm{Sing}}

\newcommand{\Gr}{\mathrm{Gr}}

\newcommand{\corank}{\mathrm{corank}}
\newcommand{\rank}{\mathrm{rank}}

\newcommand{\Pic}{\mathrm{Pic}}

\newcommand{\git}{/\kern-0.2em/}

\newcommand{\van}{\mathrm{van}}

\newcommand{\defect}{\sigma(X)}
\newcommand{\tD}{\widetilde{D}_l}

\DeclareMathOperator{\red}{red}

\author{Lisa Marquand and Sasha Viktorova}
\title[The defect of a cubic threefold]{The defect of a cubic threefold and applications to  intermediate Jacobian fibrations} 
\date{\today}

\begin{document}

\begin{abstract}
The defect of a cubic threefold $X$ with isolated singularities is a global invariant that measures the failure of $\bQ$-factoriality. We compute the defect for such cubics in terms of topological data about the curve of lines through a singular point. We express the mixed Hodge structure on the middle cohomology of $X$ in terms of both the defect and local invariants of the singularities. We then relate the defect to various geometric properties of $X$: in particular, we show that a cubic threefold is not $\bQ$-factorial if and only if it contains either a plane or a cubic scroll. We relate the defect to existence of compactified intermediate Jacobian fibrations with irreducible fibers associated to a cubic fourfold.
\end{abstract}

\maketitle
\section{Introduction}

Let $X$ be a normal complex projective variety. A key notion in studying the birational geometry of such a singular variety is the condition of $\bQ$-factoriality, i.e. the condition that every Weil divisor is a $\bQ$-Cartier divisor. Let $\mathrm{Weil}(X)$ denote the free abelian group of Weil divisors on $X$, and $\mathrm{Cart}(X)$ the subgroup of Cartier divisors, both considered with $\bQ$ coefficients. One can thus detect if a variety $X$ is $\bQ$-factorial by the vanishing of the \textbf{defect} of $X$, defined as:
$$\defect:= \dim_\bQ (\mathrm{Weil}(X)/\mathrm{Cart}(X)).$$

The purpose of this paper is to study this global invariant $\defect$ when $X$ is a cubic threefold with isolated (and thus rational) singularities. The geometry of $\bQ$-factorial Fano threefolds has been well-studied (see \cite{MR2475055}, \cite{CheltsovSextic}, \cite{MR2729277}, \cite{MR3912057}), especially in the nodal case.
A cubic threefold with only ordinary double points in general position is $\bQ$-factorial if and only if the number of points is less than 4 \cite[Theorem 1.4]{CheltsovFactorial3fold}; for other combinations of singularities, the situation is more complex. The possible combinations of isolated singularities that can occur were classified in \cite{viktorova2023classification} (see also \cite{Yokoyama}, \cite{Allcock} for a GIT perspective). 

The condition for a cubic threefold $X\subset\bP^4$ to have  isolated  singularities is equivalent to $X$ appearing as a hyperplane section of a smooth cubic fourfold $V\subset \bP^5$ \cite[Lemma 4.7]{Brosnan}. Our interest in the defect of singular cubic threefolds $X\subset \bP^4$ stems from the desire to study compactified intermediate Jacobian fibrations $\pi:\mathcal{J}_V\rightarrow (\bP^5)^*$ associated to $V$, constructed in \cite{LSV} and \cite{sac2021birational} (see also \cite{DMupcoming}). These are examples of hyperk\"ahler manifolds of $OG10$ deformation type. Briefly, for a general hyperplane $H\in (\bP^5)^*$, the fiber of $\pi$ is the intermediate Jacobian of the cubic threefold $X=H\cap V$. When $V$ is general, this Lagrangian fibration has irreducible fibers, however this is no longer true without the generality assumption. Indeed, the hyperplane sections can have worse singularities than expected, and the intermediate Jacobian of $X$ can completely degenerate. In \cite{Brosnan}, Brosnan gives a sufficient condition for a fiber of $\pi$ to be reducible: the cubic threefold obtained as a hyperplane section $X=H\cap V$ has positive defect, $\sigma(X)>0.$ 

Let $X\subset \bP^4$ be a cubic threefold with isolated singularities, that is not a cone over a cubic surface. 
Our first main result establishes that the defect $\defect$ is determined by the topology of the curve of lines through a singular point of $X$.

\begin{theorem}\label{thm:main}
    Let $X\subset \bP^4$ be a cubic threefold with isolated singularities, that is not a cone over a cubic surface. Let $q\in \Sing(X)$ be a singular point, and denote by $k$ the number of irreducible components of the curve $C_q$ parametrising lines through $q$. 
    Then:
	$$\sigma(X)=\begin{cases}
		k-1 & \text{ if } \corank(q)\neq 2;\\
		k-2 & \text{ if } \corank(q) = 2. 
	\end{cases}$$
\end{theorem}
This allows one to compute $\defect$ and the cohomology of $X$ in terms of 
the singularities of $X$, along with their relative position. 
As a by-product, we obtain a classification of $\bQ$-factorial cubic threefolds.
\begin{remark}
    If $X$ is a cone over a cubic surface, then the lines through $q$ are parametrised by this surface rather than a curve, and the defect can be computed directly from the definition. For example, the defect for a cone over a cubic surface is 6.
\end{remark}
\begin{remark} \label{choice of singularity}
    While the curve $C_q$ depends on the choice of a singular point, the invariant $\sigma(X)$ does not. Notice that the number of components $k$ of the curve $C_q$ also depends on the choice of $q$ as the threefold $X$ can contain singularities of different coranks. For example, there exist threefolds with one $D_4$ singularity (which has corank $2$) and two $A_1$ singularities (which have corank $0$). For concrete equations of such threefolds and the computation of their defect see Example \ref{D4plus2A1}.
\end{remark}

In fact, the topology of the curve $C_q$ encodes even more about the cohomology of $X$.

\begin{theorem}\label{thm: cohomology}
Let $X$ be a cubic threefold with only isolated singularities. 
Then the Hodge-Du Bois numbers of the mixed Hodge Structure on $H^3(X)$ are completely determined by a global invariant $\defect,$ and two local invariants, namely the Du Bois and link invariants of the singularities. 

Further, if $X$ is not a cone over a cubic surface, then there is an isomorphism $\Gr_3^W H^3(X)\cong \Gr_1^W H^1(C_q)(-1)$ of pure Hodge structures.
\end{theorem}

The defect of Fano threefolds and of hypersurfaces (in particular Calabi-Yau threefolds) has been studied previously in relation to their mixed Hodge structures (\cite{Clemens}, \cite{Dimca}, \cite{NS95}, \cite{Cyn}). 

Next, we wish to apply Theorem \ref{thm:main} to understand compactified intermediate Jacobian fibrations for smooth cubic fourfolds $V\subset \bP^5$. More concretely, our goal is to determine when $V$ contains a hyperplane section $X= V\cap H$ with $\defect>0.$ 
It is clear that such a hyperplane section exists if $V$ contains a plane or a rational normal cubic scroll, as observed in \cite{MR4395101}, \cite{sac2021birational}. We prove that this is an equivalence:

\begin{theorem}\label{claim}
	Let $X\subset \bP^4$ be a cubic threefold with isolated singularities. Then the following are equivalent:
	\begin{enumerate}
		\item $\sigma(X)>0$;
		\item $X$ contains either a plane, or a rational normal cubic scroll.
	\end{enumerate}
\end{theorem}

\begin{remark}
    It is possible for a cubic threefold with isolated singularities to contain surfaces of higher degree; indeed by \cite{MR1096457}, \cite{MR1191602} such a cubic can contain an elliptic quintic scroll. This is realised by the Segre cubic for instance, see \cite[Remark 1.5]{MR1096457} for more details. However, one can always find a basis for $\mathrm{Weil}(X)/\mathrm{Cart}(X)$ using only classes of planes or cubic scrolls. By Corollary \ref{defect bound} the Segre cubic has defect equal to $5$ which is maximal among the cubic threefolds with isolated singularities that are not cones, and $\mathrm{Weil}(X)/\mathrm{Cart}(X)$ is generated by classes of planes.
\end{remark}
\begin{remark}
    It was shown in \cite{MR4531678} (see also \cite{BGM25}) that cubic fourfolds $V$ with symplectic automorphisms contain either many planes, or many cubic scrolls. Such a symplectic automorphism acts on the hyperplane sections, and thus induces a birational transformation on any compactified intermediate Jacobian fibration associated to $V$ (see \cite{marquand2023classification, marquand2023finite} for more details). It follows from Theorem \ref{claim} that their associated compactified intermediate Jacobian fibrations for such symmetric cubic fourfolds have many reducible fibers, which are the obstruction for these induced birational transformations to be regular.
\end{remark}

In \cite{CML09}, the authors introduced the notion of a \textit{good line} in a cubic threefold, in order to study degenerations of the intermediate Jacobian as $X$ becomes mildly singular (see also \cite{MR4218962}). The main tool is Mumford's description of the intermediate Jacobian as a Prym variety: for a line on the cubic, projection exhibits the cubic as a conic bundle over $\bP^2$, whose discriminant curve is a plane quintic with an associated double cover. For the choice of a good line, the singularities of the quintic are in one-to-one correspondence with those of $X$, and one can study the degeneration of the intermediate Jacobian as a Prym variety of the associated double cover, even in the singular situation \cite{MR0572974}. 
In addition, the stronger notion of a \textit{very good line} was used in \cite{LSV}: for such a line, the associated quintic double cover is \'etale and both of the curves are irreducible. We show that there exists a very good line on a cubic threefold $X$, making the Prym construction possible, only when $\defect=0$.

\begin{theorem}\label{thm: no very good lines}
Let $X\subset \bP^4$ be a cubic threefold with isolated singularities. Assume further that the singularities are of corank at most 2. Then the conditions of Theorem \ref{claim} are equivalent to:
\begin{enumerate}
\setcounter{enumi}{2}
    \item $X$ does not have a very good line.
\end{enumerate}
\end{theorem}
This sharpens the results of \cite{CML09} and \cite{LSV}: there are examples of cubic threefolds with worse than allowable singularities (i.e. cubics that are not GIT semi-stable) that none the less admit a very good line. The original compactified intermediate Jacobian construction of \cite{LSV} only applied to Hodge general cubic fourfolds, i.e.. away from a countable union of divisor in the moduli of cubic fourfolds. The above result allows one to extend this compactification \cite{LSV} to a wider range of cubic fourfolds, as discussed in the recent work of \cite{DMupcoming}. 
In particular, combining \cite[Theorem 1]{DMupcoming} with Theorem \ref{claim} and \ref{no v good line} one obtains the following corollary:
\begin{corollary}
    Let $V\subset \bP^5$ be a smooth cubic fourfold, and assume there are no hyperplane sections with a corank 3 singularity. Then there exists a hyperk\"ahler compactification of the intermediate Jacobian fibration $\pi:\calJ_V\rightarrow \bP^5$ \textbf{with irreducible fibers} if and only if $X$ does not contain a plane or a rational normal cubic scroll.
\end{corollary}

\subsection*{Outline}
We recall relevant notation and definitions in \S\ref{sec: prelims}. In \S\ref{subsec: defect} we prove our formula in Theorem \ref{thm:main} for the defect in terms of components of the associated curve $C_q$ (Theorem \ref{defect}), and in \S\ref{subsec: Q fact} we provide some applications to $\bQ$-factoriality of cubic threefolds. In \S\ref{MHS} we compute the mixed Hodge structure for cubic threefolds, proving Theorem \ref{thm: cohomology} (Theorem \ref{thm: coho precise}). We prove Theorem \ref{claim} in \S\ref{defect=plane/scroll} (Theorem \ref{defect equiv to plane or scroll}), and finally discuss the existence of very good lines in \S\ref{lack of v good line}, proving Theorem \ref{thm: no very good lines} (Theorem \ref{thm: v good iff defect=0}).

\subsection*{Funding}
The second author is supported by Methusalem grant METH/21/03 -- long term structural funding of the Flemish Government. 

\subsection*{Acknowledgements}
The authors would like to thank Radu Laza for valuable advice and illuminating conversations. They would like to thank Evgeny Shinder for useful discussions about $\bQ$-factoriality and to thank him, Ivan Cheltsov, and Yagna Dutta for helpful comments on a preliminary version. They are grateful to Andreas H\"{o}ring for suggesting Corollary \ref{defect bound}. Finally, they would like to greatly thank the anonymous referees for their helpful suggestions to improve the readability of this manuscript. In particular, we thank the referee for the alternative proof in Remark \ref{ref remark}.

\section{Preliminaries}\label{sec: prelims}
Throughout, $X\subset \bP^4$ is a cubic threefold with isolated singularities that is not a cone over a cubic surface. We recall the projection method in \S \ref{subsec:project}, which we will use to compute the cohomology of $X$. In \S \ref{subsec: Vik23} we review the relevant results on singularities that appear on cubic threefolds, following \cite{viktorova2023classification}. In \S \ref{subsec: local invariants} we define the local invariants attached to the singularities of $X$.

\subsection{The projection method}\label{subsec:project} Let $q\in X$ be an isolated singularity, and choose coordinates so that $q=[0:0:0:0:1]$. The equation of $X$ is thus
$$
	f(x_0,\dots, x_4)=x_4g_2(x_0,\dots, x_3)+g_3(x_0,\dots, x_3),
$$ where $g_2, g_3$ are homogeneous polynomials of degree 2 and 3 respectively. 
Projection from the point $q$ gives a birational map $X\dashrightarrow \bP^3$; we have the diagram:
\[\begin{tikzcd}
	& Bl_qX \arrow{dl}\arrow{dr} &\\
	X\arrow[dashed]{rr} & & \bP^3
\end{tikzcd}
\]
Define $Q_q:=V(g_2(x_0,\dots, x_3))\subset \bP^3$, and $S_q:=V(g_3(x_0,\dots, x_3))$. 
We let $C_q:=Q_q\cap S_q$ be the $(2,3)$ complete intersection curve, that parametrises lines in $X$ passing through $q$. Note that if $Q_q$ is a double plane, then the curve $C_q$ is non-reduced.

\begin{lemma}\label{blow up iso}
    Let $X\subset \bP^4$ be a cubic threefold with only isolated singularities that is not a cone over a cubic surface. Assume $q\in X$ is an isolated singularity. 
    Then $$Bl_qX\cong Bl_{C_q}\bP^3.$$
\end{lemma}
\begin{proof}
	This is Proposition 3.1 of \cite{CML09}. For a more detailed proof, see \cite[Theorem 2.1.1]{Havasi}.
\end{proof}

\subsection{Singularities of cubic threefolds}\label{subsec: Vik23}

In \cite{viktorova2023classification} isolated singularities of cubic threefolds are studied using the projection method explained above. The author classified possible combinations of singularities by studying singularities of the curve $C_q$. The geometry of $C_q$ is particularly relevant to computing the cohomology of $X$.

First, we recall the necessary notation and definitions. In what follows, we are concerned with the following \textbf{analytic types} of isolated hypersurface singularities:
\begin{itemize}
    \item $A_n$, $D_n$ $(n\geq 4)$, $E_6$, $E_7$, $E_8$ ($ADE$ singularities),
    \item $U_{12}$, $S_{11}$, $Q_{10}$, $T_{pqr}$ with $\frac{1}{p}+\frac{1}{q}+\frac{1}{r}\leq1$.
\end{itemize}
For the standard local equations of these singularities we refer the reader to Chapter 15.1 of \cite{Arnold2012}.

\begin{definition}
    A \textbf{combination of singularities} on $X$ is the unordered set of all singularities on $X$. If the combination consists of singularities of types $K_1,\ldots,K_m$, we denote it as $K_1+\ldots+K_m$. We call $K_1+\ldots+K_m$ the \textbf{singularity type of} $X$.
\end{definition}

\begin{definition}
    Let $q\in X$ be a singular point. We define the \textbf{corank} of $q$ to be equal to the corank of the quadric $Q_q.$ Recall that the corank of the quadric $Q_q\subset \bP^3$ is the corank of the corresponding quadratic form $g_2$, that is, $4-\rank (g_2)$.
\end{definition}
This definition agrees with the usual definition of corank of a critical point of a function in this case. 
We will use the following well-known lemma (see e.g. \cite[Theorem 16.2.2]{Arnold2012}):

\begin{lemma} \label{An sing}
    The corank of $q$ is equal to one if and only if $q$ is of $A_r$ type with $r>1.$ The corank of $q$ is equal to zero if and only if $q$ is an $A_1$ singularity.
\end{lemma}

\begin{proposition} [{\cite{viktorova2023classification}, Proposition 2.2, Table 2}]
    Suppose a singular point $q\in X$ is of corank $3$. Then $q$ 
     is the only singularity of $X$ and its type is $U_{12}$, $S_{11}$, $Q_{10}$, $T_{444}$, $T_{344}$, $T_{334}$ or $T_{333}$. Additionally, the curve $C_q$ is irreducible if and only $q$ is a $T_{333}$, $T_{334}$ or $Q_{10}$ singularity.
\end{proposition}

\begin{proposition} [{\cite{viktorova2023classification}, Proposition 3.5, Figure 4, Table 3}] \label{corank 2}
    Let $q\in X$ be a singular point of corank~$2$. Then $Q_q$ is a union of two distinct planes and $C_q$ is a union of two plane curves. Thus $C_q$ has at least two irreducible components. Moreover, the following statements hold:
    \begin{enumerate}
        \item If $C_q$ has exactly two irreducible components, then $q$ is of type $D_n$, $4\leq n\leq 7$ or $E_n$, $n=6, 7, 8$.
        \item If $q$ is a $D_7$, $D_8$, $E_8$ or a non-$ADE$ singularity, then the number of components of $C_q$ is determined by the singularity type of $X$.
    \end{enumerate}
\end{proposition}

For the rest of this subsection, we assume that $X$ only contains singularities of corank $1$ and $0$ which are $A_n$ singularities by Lemma \ref{An sing}. Let $q\in X$ be an $A_n$ singularity with the largest $n$. If $n>1$, then the blow up $\widetilde{Q}_q$ of the cone $Q_q$ at the cone point is isomorphic to the Hirzebruch surface $\bF_2$. We denote by $\widetilde{C}$ the strict transform of $C_q$.

\begin{lemma} [{\cite{viktorova2023classification}, Lemma 4.4, Propositions 4.6-4.8}] \label{classofc}
	Assume that $q$ is of $A_n$ type with $n>2$. Then the class $[\widetilde{C}]\in\Pic(\bF_2)$ equals $2\sigma+6f$, where $\sigma$ is the class of the unique irreducible curve with negative self-intersection and $f$ is the class of a fiber. Then the only possible classes of irreducible components of $\widetilde{C}$ are $2\sigma+6f$, $2\sigma+5f$, $2\sigma+4f$, $\sigma+4f$, $\sigma+3f$, $\sigma+2f$, or $f$. In particular, the class $\sigma+3f$ corresponds to a twisted cubic on $C_q$, $\sigma+2f$ corresponds to a conic and $f$ to a line.
\end{lemma}

\begin{lemma} [{\cite{viktorova2023classification}, Lemma 4.5, Proposition 4.9}] \label{classofc A2}
    Assume that $q$ is of $A_2$ type. Then the class  $[\widetilde{C}]=3\sigma+6f\in\Pic(\bF_2)$. The only possible classes of irreducible components of $\widetilde{C}$ are $3H$, $2H$ and $H$, where $H=\sigma+2f\in\Pic(\bF_2)$ is the class of the preimage of a hyperplane section of $Q_q$ not intersecting the cone point.
\end{lemma}

\begin{corollary} \label{components An} Let $q\in X$ be a singular point of type $A_n$ with $n>1$. 
    Then $C_q$ has at most $4$ components. Further, if the curve $C_q$ is reducible, then it contains either a twisted cubic, a conic or a line.
\end{corollary}

\subsection{Local invariants}\label{subsec: local invariants}

The cohomology of a compact complex algebraic variety $X$ carries a mixed Hodge structure $(H^{p+q}(X,\bZ), W_\bullet, F^\bullet)$, where $W_\bullet$ is the (increasing) weight filtration and $F^\bullet$ is the (degreasing) Hodge filtration. The graded pieces $\Gr^W_kH^{p+q}(X)=W_kH^{p+q}(X)/W_{k-1}H^{p+q}(X)$ carry a pure Hodge structure of weight $k$ induced by the Hodge filtration $F$.
We define
the \textbf{Hodge-Du Bois numbers} by:
    \[\underline{h}^{p,q}(X):= \dim\Gr_F^pH^{p+q}(X).\]
If $X$ is smooth they recover the usual Hodge numbers $\underline{h}^{p,q}(X)=h^{p,q}(X)$. In order to understand the mixed Hodge structure for cubic threefolds with isolated singularities, we introduce some local invariants as follows.

Throughout, we let $(X,x)$ be a germ of an isolated threefold singularity and $\psi: (Y,E)\rightarrow (X,x)$ its resolution with $E$ a normal crossing divisor. We fix a representative $\psi: Y\rightarrow X$ such that $X$ is a contractible Stein space. We introduce the two local invariants, namely the Du Bois and link invariants of \cite{duBoisInv}. Since we will be considering singular cubic threefolds whose isolated singularities are rational,
we will use the properties of rational singularities to simplify the computation of local invariants where possible.

\begin{definition}\cite{duBoisInv}
    We define the \textbf{Du Bois invariants} of $(X,x)$ as follows: for $p\geq 0, q>0$ we have
	$$b^{p,q}=b^{p,q}(X,x):=\dim H^q(Y, \Omega^p_Y(\log E)(-E)).$$
\end{definition}

Recall that $(X,x)$ is a \textbf{Du Bois singularity} if $b^{0,q}=0$ for $0<q<3$, and $(X,x)$ is a \textbf{rational singularity} if $H^i(Y,\calO_Y)=0$ for $0<i<3$. Note that rational implies Du Bois.

\begin{lemma}
	Let $(X,x)$ be a germ of an isolated, rational, threefold singularity. Then $b^{p,q}=0$ except for possibly $b^{1,1}$ and $b^{2,1}$.
\end{lemma}
\begin{proof}
	By assumption, $b^{0,q}=0$ for $0<q<3$, and by \cite[Theorem 1]{duBoisInv} $b^{p,q}=0$ for $p+q>3$ or $q\geq 3$. We also have $b^{1,2}=0$ for a Du Bois singularity (see for instance \cite[p. 1369]{duBoisInv}).
\end{proof}

Let $L_x$ be the link of the singularity $(X,x)\subset (\bC^4,0),$ i.e. the boundary of an algebraic neighbourhood of the singularity $(X,x)$ (see e.g. \cite[Section 6.2]{mixedhodge}).
The cohomology $H^i(L_x)$ also carries a mixed Hodge structure. Following \cite{duBoisInv}, we use the Hodge filtration to define the following \textbf{link invariants:}

\begin{definition}
 We define \textbf{link invariants} of the singularity $(X,x)$ as
$$ l^{p,q}=l^{p,q}(X,x):=\dim H^q(E, \Omega^p(\log E)\otimes \calO_E)= \dim \Gr^p_FH^{p+q}(L_x).$$
\end{definition}
There is a duality of mixed Hodge structures between $H^i(L_x)$ and $H^{5-i}(L_x)$ by \cite[Corollary 1.15]{Ste83}. It follows that $$l^{p,q}=l^{3-p, 2-q}.$$

\begin{definition}
	The \textbf{Deligne-Hodge numbers} of the link $L_x$ are given by
	$$h_i^{p,q}(L_x)=\dim \Gr_F^p\Gr^W_{p+q}H^i(L_x).$$
\end{definition}
We see that $h_i^{p,q}(L_x)=h_i^{q,p}(L_x),$ and $l^{p,i-p}=\sum_q h_i^{p,q}(L_x)$. Further, by the semi-purity theorem of Goresky-MacPherson \cite[Theorem 1.11]{Ste83} we have $h_i^{p,q}(L_x)=0$ if $i<3$ and $p+q>3$. 

We note:
\[l^{p,q}=\sum_{0\leq r\leq q} h^{p,r}_{p+q} (L_x).\]
\begin{lemma} \label{link vanishing}
	Let $(X,x)$ be a rational isolated local complete intersection threefold singularity. Then $l^{p,q}=0$ except for possibly $l^{1,1}$.
\end{lemma}
\begin{proof}
	The link of a $3$-dimensional isolated complete intersection singularity is simply connected (see e.g. \cite[Section 4]{duBoisInv}), so $H^i(L_x)\neq 0$ only for $i\in \{0, 2, 3, 5\}$. It follows in this case that $l^{p,q}=0$ unless $p+q\in\{0, 2, 3, 5\}.$  Since $(X,x)$ is a rational singularity, by \cite[Lemma 2]{duBoisInv} we have $l^{0,i}=l^{i,0}=0$ for all $i$.
	\end{proof}

The invariants $l^{p,q}$ and $b^{p,q}$ can be related via the cohomology of the Milnor fiber $M_x$ of the singularity $(X,x)$.
We recall the definition below:

\begin{definition}
    Let $(X,x)$ be an isolated threefold hpersurface singularity, with local equation $f:\bC^4\rightarrow \bC$. Pick sufficiently small neighbourhoods $x\in U \subset \bC^4$ and
$f(x) \in T \subset \bC$. Then the Milnor fiber of $(X,x)$ is defined to be the hypersurface $M_{x} := f^{-1}(t) \cap U$ for some $t\in T$, $t\neq f(x)$.

The Milnor number $\mu(X,x)$ of the singularity $(X,x)$ is defined as $h^3(M_x)$.
\end{definition}

Note that all fibers $f^{-1}(t) \cap U$ for some $t\in T$, $t\neq f(x)$ are diffeomorphic to each other, so $M_{x}$ and $\mu(X,x)$ are well defined invariants of $(X,x).$

For an isolated hypersurface threefold singularity, and more generally for an isolated local complete intersection threefold singularity, $H^3(M_x)$ admits a mixed Hodge structure, and we set
 $$s_p:=\dim \Gr^p_F H^3(M_x).$$
The invariants are related by the following theorem:

\begin{theorem}\cite[Theorem 6]{duBoisInv}\label{milnor fiber Hodge numbers}
	Let $(X,x)$ be a three-dimensional Du Bois isolated complete intersection singularity with $l^{2,0}=0$. Then the Hodge numbers of the Milnor fiber are given by:
	$$s_0=0,\,\, s_1=b^{1,1},\,\, s_2=b^{1,1}+l^{1,1},\,\, s_3=l^{0,2}.$$
\end{theorem}

\section{The cohomology of $X$}\label{cohomology}
We compute the cohomology for a cubic threefold $X\subset\bP^4$ with isolated singularities (that is not a cone over a cubic surface).
First we define the defect $\sigma(X)$ and provide in \S\ref{defect} a simple way to compute this invariant by counting irreducible components of the $(2,3)$ complete intersection curve $C_q$ associated to a singular point $q\in X.$ 
In \S\ref{MHS}, we compute the mixed Hodge structure on $H^3(X,\bZ)$ in terms of local invariants of the singularities, as well as $\sigma(X).$
Further, we show that the mixed Hodge structure on $H^3(X,\bZ)$ is completely determined by that on $H^1(C_q, \bZ)$.

\subsection{The defect of $X$}\label{subsec: defect}
\begin{definition}
    Let $X$ be a normal projective variety. 
    The \textbf{defect} $\sigma(X)$ of $X$ is the dimension of the finitely generated abelian group $\mathrm{Weil}(X)/\mathrm{Cart}(X)$.
\end{definition}
When $X$ is a projective threefold with isolated rational singularities, by \cite[Theorem 3.2]{NS95}, one can calculate the defect by $$\sigma(X)=b_4(X)-b_2(X).$$  
We provide a geometric way to compute the defect for a cubic threefold with isolated singularities.

\begin{theorem}[Theorem \ref{thm:main}]\label{defect}
	Let $X\subset \bP^4$ be a cubic threefold with only isolated singularities, that is not a cone over a cubic surface. 
    Let $q\in X$ be such a singularity, and let $k$ denote the number of irreducible components of the associated $(2,3)$ complete intersection curve $C_q$. Then:
	$$\sigma(X)=\begin{cases}
		k-1 & \text{ if } \corank(q)\neq 2;\\
		k-2 & \text{ if } \corank(q) = 2. 
	\end{cases}$$
\end{theorem}
\begin{proof}
	By Lemma \ref{blow up iso}, $\widetilde{X}:=Bl_qX\cong Bl_{C_q}\bP^3$, and so we have two discriminant squares:	
	\[\begin{tikzcd}
		(a) &Q\arrow{r}{j}\arrow{d}{g} &\widetilde{X}\arrow{d}{f} & &(b)&E\arrow{r}{j}\arrow{d}{g} &\widetilde{X}\arrow{d}{f} \\
		&    q\arrow{r}{i} &X&&&C\arrow{r}{i} &\bP^3
	\end{tikzcd}\]
	Here, $Q$ is the reduced subscheme of $\widetilde{X}$ associated to the projectivised tangent cone of the singularity $q$ and is isomorphic to $(Q_q)_{\red}$. Thus $Q$ is either a smooth quadric, a quadric cone, the union of two planes intersecting in a line, or a plane, depending on the corank of the singularity $q$. Let us record that $h^4(Q)=2$ if $\corank(q)=2$, and $h^4(Q)=1$ otherwise.
	
	Using the isomorphism $\widetilde{X}\cong Bl_{C_q}X$, the exceptional divisor $E$ of $f:\widetilde{X}\rightarrow \bP^3$ is a projective bundle over $C:=(C_q)_{\red}$. 
    Indeed, since $C$ is a complete intersection the conormal bundle $\calC_{C/\bP^3}:=\calI/\calI^2$ of $C$ is locally free and $E=\Proj(\Sym(\calI/\calI^2))$. 
    Let $h:=\calO_E(1)\in H^2(E,\bZ)$; the classes $1, h$ when restricted to a fiber $E_x $ of $E\rightarrow C$ form a basis of $H^*(E_x,\bZ)$.
    By \cite[Theorem 7.33]{Voi1}, $H^*(E,\bZ)$ is isomorphic to $A\otimes H^*(C,\bZ)$, where $A\subset H^*(E,\bZ)$ is the subgroup generated by $1, h$. 
     In particular, $$H^3(E,\bZ)\cong h\cdot g^*H^1(C,\bZ),$$ via pullback and the cup-product, both of which are morphisms of mixed Hodge structures.
	
	Now we apply the Mayer-Vietoris sequence (see \cite[Corollary-Definition 5.37]{mixedhodge}) to $(b)$:
	$$0\rightarrow H^3(\widetilde{X})\rightarrow H^3(E)\rightarrow H^4(\bP^3)\rightarrow H^4(\widetilde{X})\rightarrow H^4(E)\rightarrow 0.$$
	Since $H^*(E,\bZ)\cong\langle 1, h\rangle\otimes H^*(C,\bZ)$ as above, we conclude that $h^3(E)=h^1(C)$ and $h^4(E)=h^2(C)$. 
    Since $H^4(\bP^3)$ has a pure Hodge structure of weight $4,$ and $H^3(E)$ has weights at most $3$, it follows that
	\begin{align*}
		h^3(\widetilde{X})&=h^3(E)=h^1(C),\\
		h^4(\widetilde{X})&=h^2(C)+1.
	\end{align*}
    We apply the same sequence to $(a)$; since $H^3(Q)=0$ this reduces to the short exact sequence $$0\rightarrow H^4(X)\rightarrow H^4(\widetilde{X})\rightarrow H^4(Q)\rightarrow 0.$$
    Combining these facts we see that 
	$h^4(X)=k+1-h^4(Q)$; since $h^2(X)=1$, the claim follows.
\end{proof}

\begin{corollary} \label{defect bound}
    Let $X\subset \bP^4$ be a cubic threefold with only isolated singularities. Then $\defect\leq 6$. Further, $\defect= 6$ if and only if $X$ is a cone over a cubic surface, and $\defect=5$ if and only if $X$ is the Segre cubic.
\end{corollary}

\begin{proof}
    It is clear that if $X$ is a cone over a cubic surface, then $\defect=6;$ indeed, the generators of the group $\mathrm{Weil}(X)/\mathrm{Cart}(X)$ are the planes obtained as cones over the lines on the cubic surface.
    Suppose then that $X$ is not a cone: if $q\in X$ is a node, then $Q_q$ is smooth and $C_q$ has at most 6 components. This is only possible if $C_q$ has 9 additional nodes: in other words, $X$ is the 10 nodal cubic, i.e. the Segre cubic. If $q\in X$ is instead a singularity of type $A_n$ with $n>1$, then $C_q$ has at most 4 components by Corollary \ref{components An}.
    
    The curve $C_q$ corresponding to a corank 2 singularity is a union of two plane cubics, thus it has at most 6 components, and $\defect\leq 6-2=4$. The curve $C_q$ corresponding to a corank 3 singularity has at most 3 components since $Q_q$ is a double plane.
\end{proof}

\subsection{$\bQ$-factoriality}\label{subsec: Q fact}
As an application of Theorem \ref{defect}, and utilising the classification of possible singularities of a cubic threefold $X$, we are able to determine which cubic threefolds with isolated singularities are $\bQ$-factorial. All of these results directly follow from studying the curve $C_q$ and counting the number of irreducible components.

\begin{proposition}
    Let $X$ be a cubic threefold with isolated singularities, and assume that one of the singular points of $X$ has corank greater than 1.
    \begin{enumerate}
        \item If $q\in X$ is a corank 3 singularity, then $X$ is $\bQ$-factorial if and only if $q$ is a $T_{333}$, $T_{334}$ or $Q_{10}$ singularity.
        \item If $q\in X$ is a corank 2 non-$ADE$ singularity, then $\defect>0$. 
    \end{enumerate}
    Moreover, in all the cases above $\defect$ depends only on the singularity type of $X$, or in other words, $\defect$ is determined by the local data.
\end{proposition}
\begin{proof}
    This follows immediately from Theorem \ref{defect} and Proposition \ref{corank 2}.
\end{proof}

For combinations of $ADE$ singularities, $\defect$ depends on the singularity type and the geometry of $X$, but can still be determined by studying the curve $C_q$. We will focus on the $\bQ$-factoriality of maximal combinations of singularities -- recall the following definition:
\begin{definition}
    Let $\calT$ be a collection of combinations of isolated singularities appearing on cubic threefolds. We say that a combination $\mathcal{K}$ is \textbf{maximal} in the collection $\calT$ if there is no family $\calX\rightarrow\bA^1$ of cubic threefolds such that it has a fiber with a combination of singularities $\mathcal{K}'\neq\mathcal{K}$, $\mathcal{K}'\in\calT$, while the generic fiber has the combination $\mathcal{K}$.
\end{definition}

For any combination of singularities $T$ appearing on a cubic threefold $X$, one can easily determine whether $X$ is $\bQ$-factorial or not by using \cite{viktorova2023classification} (see for example \cite{CMTZ} for explicit computations of $\sigma(X)$). We illustrate this below for maximal combinations of singularities.

\begin{proposition}
    A cubic threefold with $E_8+A_2$ or $D_5+2A_2$ singularities is $\bQ$-factorial.
    A cubic threefold with $E_7+2A_1$ is $\bQ$-factorial if and only if it does not contain a plane. Moreover, these are the maximal combinations that can be $\bQ$-factorial among the combinations containing a $D_n$ or $E_n$ singularity.
\end{proposition}
\begin{proof}
    This is essentially contained in \cite[Section 3]{viktorova2023classification}; we omit the details for brevity.
\end{proof}

We illustrate the dependence on the geometry of $X$ with the example below:
\begin{example} \label{D4plus2A1}
    A cubic threefold $X$ with a $D_4+2A_1$ combination of singularities can have $\defect=0$ or $1$. Explicitly, let $q=[1:0:0:0:0]\in X$ be of type $D_4$. By \cite[Section 3.2]{viktorova2023classification}, we have that $Q_q=P_1\cup P_2$ for planes $P_i$ intersecting in a line $L$. The curve $C_q$ is a union of two plane cubics $C_1$ and $C_2$ which intersect at three distinct points on $L$. If $C_i$ are both irreducible with one node, then $\defect=0$ by Theorem \ref{defect}. If $C_1$ is smooth and $C_2$ has two nodes (and thus reducible), then $\defect=1$. Both situations occur, for explicit cubics below:
    \begin{align*}
    \defect&=0: \,\, x_0x_1x_2+x_1^2x_4-2x_1x_4^2+x_2^2x_4-2x_2x_4^2-x_3^3-x_3^2x_4+x_4^3=0;\\
    \defect&=1: \,\, x_0x_1x_2+x_1x_3x_4+x_2^3-x_3^3+x_3x_4^2=0.
    \end{align*}

    We will use the second equation to further illustrate Remark \ref{choice of singularity}, that the $\defect$ is independent of the choice of singular point to project from. Notice that we can rewrite this equation as
    $$
    x_1(x_0x_2+x_3x_4)+x_2^3-x_3^3+x_3x_4^2=0.
    $$
    We see that $q'=[0:1:0:0:0]$ is a singularity of corank $0$ and thus a node. Set $Q_{q'}=V(x_0x_2+x_3x_4=0)\subset\bP^3$ and $S_{q'}=V(x_2^3-x_3^3+x_3x_4^2=0)\subset\bP^3$. The curve $C_{q'}$ has two irreducible components: one is the line $V(x_2=x_3=0)$, the other is the closure of the affine curve $V(x_0x_2+x_4=x_2^3+x_4^2-1=0)$ in the chart $x_3\neq 0$. We again conclude that the defect is $1$ in this case.
\end{example}

\begin{proposition}
    A cubic threefold with $A_{10}$, $A_8+A_2$, $A_6+A_4$, $2A_4+A_2$ or $5A_2$ singularities is $\bQ$-factorial. Moreover, these are the maximal combinations of $A_n$ singularities that can be $\bQ$-factorial.
\end{proposition}
\begin{proof}
    The detailed case by case analysis is contained in the proofs in \cite[Section 4]{viktorova2023classification}. The starting point is Lemmas \ref{classofc} and \ref{classofc A2}.
\end{proof}

\begin{remark}
    Notice that the defect is not upper semi-continuous in families. For example, a cubic threefold with an $A_{10}$ singularity has defect $0$ while a cubic threefold with an $A_9$ singularity has defect $1$ by Theorem \ref{defect} and \cite[Propositions 4.7 and 4.8]{viktorova2023classification}. By a theorem of du Plessis and Wall \cite[Proposition 4.1]{DPW}, there exists a family of cubic threefolds with the generic fiber having an $A_9$ singularity, whereas the special fiber has one $A_{10}$ singularity.
\end{remark}

\subsection{Mixed Hodge structure on $H^3(X,\bZ)$}\label{MHS}

To completely understand the cohomology of the cubic threefold $X$ with isolated singularities, it remains to understand the mixed Hodge structure on the middle cohomology. Let $\{q_1,\dots q_r\}\in X$ be the isolated singular points of $X$. We denote by $$B:=\sum_{i=1}^rb^{1,1}(X,q_i) \text{ and } L:=\sum_{i=1}^r l^{1,1}(X, q_i)$$ the total Du Bois and link invariants of $X$. 
Note that the invariants $b^{1,1}$ and $l^{1,1}$ are the only Du Bois and link invariants needed to compute the Hodge numbers of the Milnor fiber by Theorem \ref{milnor fiber Hodge numbers}.  We will use this theorem along with the vanishing cycle sequence to prove the following:

\begin{theorem}\label{thm: coho precise}
	Let $X$ be a cubic threefold with isolated singularities which is not a cone over a cubic surface. Then the non-zero Hodge-Du Bois numbers of the mixed Hodge structure on $H^3(X)$ are computed as follows:
\begin{align*}
    \underline{h}^{1,2}(X)&= 5-B; \\
    \underline{h}^{2,1}(X)&= 5-(L-\sigma)-B. 
\end{align*}
\end{theorem}
\begin{proof}
Recall that $X$ has rational hypersurface singularities. Since $X$ is a compact variety, it follows that $\Gr^W_kH^3(X)=0$ for $k>3$.  We let $\mathcal{X}\rightarrow\Delta$ be a $1$-parameter smoothing of $X$, and consider the vanishing cycle sequence:

\[0\rightarrow H^3(X)\rightarrow H^3_{\lim}(X_t)\rightarrow H^3_{\van}\rightarrow H^4(X)\rightarrow H^4_{\lim}(X_t)\rightarrow 0.\]
 Let us first consider the limit cohomology $H^3_{\lim}(X_t)$.  For fixed $p$, we have that $$\sum_qh^{p,q} ( H^3_{\lim}(X_t))=h^{p,3-p}(X_\infty),$$ where $X_\infty$ is a smooth cubic threefold. Denote by: 
 $\alpha:=h^{1,1}(H^3_{\lim}(X_t)) \text{ and } \beta:=h^{1,2}(H^3_{\lim}(X_t)).$ Recall that the Hodge numbers of the middle cohomology of a smooth cubic threefold $X_\infty$ are $h^{2,1}(X_\infty)=h^{1,2}(X_\infty)=5$; it follows that $\alpha+\beta =5$. 

 The vanishing cohomology $H^3_{van}$ has a mixed Hodge structure induced through the isomorphism $$H^3_{van}\cong \bigoplus H^3(M_{q_i}),$$ where $M_{q_i}$ is the Milnor fiber of the singularity $q_i$.  We can thus compute the Hodge numbers of $H^3_{van}$ by using Theorem \ref{milnor fiber Hodge numbers}. We find that $H^{2,1}(H^3_{\van})=B,$ $H^{2,2}(H^3_{\van})=L$ and $H^3_{\van}$ has no weight 2 cohomology. It follows that $h^{1,1}(H^3(X))=\alpha.$
 
 On the other hand, the map $H^3_{van}\rightarrow H^4(X)$ is a morphism of pure weight 4 Hodge structures by \cite[Prop 5.5]{KLS19}.  Since $\dim\ker(H^4(X)\rightarrow H^4_{\lim}(X_t))=\sigma(X)$ and is pure of type $(2,2)$, by comparing weight 4 cohomology we have that $\alpha=L-\sigma$. 
	Comparing weight three cohomology gives $\beta=h^{3,1}(H^3(X))+B$, and the result follows.
\end{proof}

\begin{corollary}
    If $X$ is a cubic threefold with isolated singularities, then $$h^3(X)=10-\mu_{tot}(X)+\defect,$$
    where $\mu_{tot}(X)=\sum_{q\in \Sing(X)}\mu(q)$ is the sum of the Milnor numbers of all the singular points of $X$.
\end{corollary}
\begin{proof}
    The statement follows immediately from Lemma \ref{link vanishing}, Theorem \ref{milnor fiber Hodge numbers}, and Theorem \ref{thm: coho precise} above. For the cone case, it is known that $h^3(X)=0$ (see e.g. \cite[Theorem 2.1]{Dim86}), $\defect=6$ by Corollary \ref{defect bound} and the Milnor number is 16.
\end{proof}

We've seen in Theorem \ref{defect} that the topology of the curve $C_q$ determines the global invariant $\sigma(X)$; in fact, the local invariants $B, L$ are also determined by the singularities of $C_q$. More precisely:
\begin{proposition}\label{Gr} 
	There is an isomorphism $\Gr_3^WH^3(X)\cong \Gr_1^WH^1(C_q)(-1)$ of Hodge structures.
\end{proposition}

\begin{proof}
	Let us apply the Mayer-Vietoris sequence to the discriminant square $(a)$ of Theorem \ref{defect}:
	\[\begin{tikzcd}
		0\arrow[r]&H^2(X)\arrow[r]&H^2(\widetilde{X})\arrow[r]& H^2(Q)\arrow[r]&H^3(X)\arrow[r]&H^3(\widetilde{X})\arrow[r]&0.
	\end{tikzcd}\]
	Since $X$ is a hypersurface with isolated singularities, $H^2(X)$ has a pure weight 2 Hodge structure \cite[Lemma 1.5]{KL20}, and $h^2(X)=1$. Since $\widetilde{X}=Bl_qX$, it follows that $h^2(\widetilde{X})=2$. We see that the sequence $H^2(X)\rightarrow H^2(\widetilde{X})\rightarrow H^2(Q)$ is an exact sequence of pure weight 2 Hodge structures; it follows that we have an isomorphism of mixed Hodge structures:
	$H^3(X)\cong H^3(\widetilde{X}).$
	Now let us apply the Mayer-Vietoris sequence to $(b)$:
	\[\begin{tikzcd}
		0\arrow[r]&H^3(\widetilde{X})\arrow[r]& H^3(E)\arrow[r]&H^4(\bP^3)\arrow[r]& H^4(E)\arrow[r]&0.
	\end{tikzcd}\]
	Notice that $H^4(\bP^4)$ is a pure Hodge structure of weight $4$; since $H^3(E)$ has weights at most 3, it follows that we have an isomorphism of mixed Hodge structures: $H^3(\widetilde{X})\cong H^3(E).$
	In particular, $Gr_3^WH^3(X)\cong Gr_3^WH^3(E).$ We saw in the proof of Theorem \ref{defect} that  $H^3(E,\bZ)\cong h\cdot f^*H^1(C_q,\bZ),$ via pullback and the cup-product, both of which are morphisms of mixed Hodge structures. In particular, it follows that we have an isomorphism of Hodge structures $Gr_3^WH^3(E)\cong Gr_1^WH^1(C_q)(-1).$
\end{proof}

One can compute the non-zero Hodge Du Bois numbers of the mixed Hodge structure on $H^1(C_q)$ in terms of the local invariants of the singularities of $C_q$, and in turn can compute both $B$ and $L$.

\section{Planes and rational normal cubic scrolls}\label{defect=plane/scroll}
Throughout, let $X$ be a cubic threefold with isolated singularities that is not a cone over a cubic surface. We will prove the following theorem:

\begin{theorem}[Theorem \ref{claim}]\label{defect equiv to plane or scroll}
	Let $X$ be a cubic threefold with isolated singularities that is not a cone over a cubic surface. Then the defect $\sigma(X)>0$ if and only if $X$ contains a plane or a rational normal cubic scroll.
\end{theorem}
Recall that a rational normal cubic scroll is a a smooth connected 
surface  $\Sigma\subset \bP^4$ of minimal degree with $\text{span}(\Sigma)=\bP^4$. Such a surface is isomorphic to $\bP^2$ blown up in a point.
 It is clear that if $X$ contains a plane or a cubic scroll, then $\sigma(X)>0$ -- we include a proof for completeness. 

\begin{lemma} \label{plane => not palindromic}
	If a cubic threefold $X$ contains a plane or a cubic scroll, then the defect $\sigma(X)>0.$
\end{lemma}
\begin{proof}
    Suppose that $\defect=0$; then every Weil divisor on $X$ is Cartier, and in particular for every Cartier divisor $Z\subset X$ there is a hypersurface $H\subset \bP^4$ with $Z=X\cap H$ (see for \cite[Remark 1.2]{CheltsovFactorial3fold} for instance). This is a contradiction: indeed, a plane can never be obtained as such an intersection, and a rational normal scroll is a nondegenerate surface of degree $3$. 
\end{proof}

To prove the converse we again will exploit the geometry of the associated curve $C_q$. Combining Lemmas \ref{lem: nodes}, \ref{lem: An}, \ref{lem: corank 2} below completes the proof of Theorem \ref{defect equiv to plane or scroll}.

\begin{lemma}\label{lem: nodes}
	Let $X$ be a cubic threefold with only isolated nodal singularities, and suppose $\sigma(X)>0$. Then $X$ contains a plane or a rational normal cubic scroll.
\end{lemma}
\begin{proof}
    The case of nodes is essentially contained in \cite{Finkelnberg}; we provide a proof using the language of this note for completion.
    Let $q\in X$ be a node, and project to obtain $C_q\subset \bP^3$. The curve $C_q$ lies on a smooth quadric, and the singularities of $C_q$ are in one-to-one correspondence with the remaining singularities of $X$. Since $\sigma(X)>0$, the curve $C_q$ is reducible by Theorem \ref{defect}. If $X$ has at least four nodes contained in a plane, then the plane is contained in $X$ by \cite[Lemma 2.10]{Avilov}, and we are done. Thus it suffices to consider $X$ with nodes in general position. In this case, we see that $C_q$ has at least $5$ nodes, thus $X$ must have at least $6$ nodes. If the number of nodes is strictly greater than $6$, the curve $C_q$ either contains a line, or a plane conic -- both imply the existence of a plane in $X$. 
    If $X$ has exactly $6$ nodes in general position, then $X$ contains a cubic scroll by \cite[Prop. 23]{flops}. 
\end{proof}

\begin{lemma}\label{lem: An}
	Let $X$ be a cubic threefold with a combination of $A_n$ singularities, with at least one with $n\geq 2$. Assume that $\sigma(X)>0$. Then $X$ contains a plane or a rational normal cubic scroll.
\end{lemma}
\begin{proof}
    Let $q\in X$ be the worst singularity (i.e $A_n$ for $n\geq 2)$; the curve $C_q\subset \bP^3$ lies on the quadric cone $Q_q$. The blow up $\widetilde{Q}_q$ of $Q_q$ at the cone point is isomorphic to $\bF_2$; let $\widetilde{C}\subset \bF_2$ be the strict transform of $C_q\subset Q_q$. Applying Theorem \ref{defect}, the curve $\widetilde{C}$ is reducible. By Corollary \ref{components An}, whenever $\widetilde{C}$ is reducible, it contains either the preimage of a ruling $\gamma$ of $Q_q,$ the preimage of a hyperplane section $\Gamma$ of $Q_q$ missing the cone point, or the preimage of two twisted cubics $\Gamma_1, \Gamma_2\subset \bP^3$.
  
    In the first two cases, $X$ clearly contains a plane; it suffices to consider the latter. We claim that $X$ contains a rational normal cubic scroll. Indeed, $\Gamma_1, \Gamma_2$ are two twisted cubics lying on the cubic surface $S_q$. Following the argument of \cite[Thm 3.5]{dolgnodal}, the linear systems $|\Gamma_1|, |\Gamma_2|$ define two determinantal representations of the surface $S_q$, each obtained from the other by taking the transpose of the matrix. In an identical way to \cite[Thm 3.5]{dolgnodal}, the cubic threefold $X$ is isomorphic to a hypersurface $V(\det(A))$, where $A$ is a $3\times 3$ matrix with linear forms in coordinates on $\bP^4$. By \cite[Prop 17]{flops}, it follows that $X$ contains a rational normal cubic scroll (see also \cite[Lemma 4.1.1]{hassett}).
\end{proof}

\begin{lemma}\label{lem: corank 2}
	Let $X\subset\bP^4$ be a cubic threefold that is not a cone over a cubic surface. Assume that there exists at least one singularity $q\in X$ of corank $\geq 2$. Assume that $\sigma(X)>0.$ Then $X$ contains a plane.
\end{lemma}
\begin{proof}
	We project from the point $q$ - the quadric surface $Q_q$ is either the union of two planes $P_1\cup P_2\subset \bP^3$, or a double plane. In the first case, $C_q=C_1\cup C_2$, where $C_i\subset P_1$ are plane cubics, whereas in the second $C_q$ is a double cubic. Since $X$ has positive defect, in both cases it follows that $C_q$ contains a line component $L$, and $X$ contains a plane spanned by the lines through $q$ parametrised by $L$.
\end{proof}

\begin{remark}\label{ref remark}
    We note the following alternative proof that $\defect>0$ implies that $X$ contains a plane or a scroll, as suggested by the referee.

    We first note that the singularities of $X$ are terminal, since $X$ is not a cone.
    Since $\sigma(X)>0$, the variety $X$ has a non-trivial $\mathbf Q$-factorialization $X'\to X$.
One applies the classification of possible Mori contractions (see for example \cite{Cut}), for a cubic threefold and sees that
one of the following holds: 
\begin{enumerate}
    \item there is a fibration $X'\to \mathbf P^1$ whose general fiber is a quadric,
    \item there is a birational contraction whose exceptional divisor is a plane,
    \item there is a $\mathbf P^1$-bundle structure $X'\to \mathbf P^2$ so that the pull-back of a line is a cubic scroll.
\end{enumerate}

\end{remark}

\section{Very good lines}\label{lack of v good line}
In this section we relate Theorem \ref{claim} to another geometric property of the cubic threefold $X$, namely the non-existence of a very good line (see Definition \ref{defn: v good}).
In particular, we will prove:
\begin{theorem}\label{thm: v good iff defect=0}
    Let $X$ be a cubic threefold with isolated singularities. Assume $X$ does not have a singularity of type $T_{333}, T_{334}, Q_{10}$, and is not a cone over a cubic surface. Then $X$ has a very good line if and only if $\sigma(X)=0.$
\end{theorem}
This proves Theorem \ref{thm: no very good lines}: the excluded singularities are those of corank $3$ for which $\defect=0.$ We outline the structure of this section:
in \S \ref{v good lines} we will recall the necessary definitions. In particular, we recall the existence of good lines for cubic threefolds with $\defect=0$. Next in \S \ref{subsec: plane or scroll=> no v good line} we will prove that a cubic containing a plane or a rational normal cubic scroll does not contain a very good line, proving one direction of Theorem \ref{thm: v good iff defect=0}. The remaining direction is proved in \S \ref{no v good line}: let us outline the strategy. We prove by contradiction, and suppose that $\defect=0$ but none of the good lines $l\subset X$ are very good, i.e $\tD$ is reducible. There are two possibilities: either the plane quintic $D_l=\tD/\tau$ is reducible or irreducible. We first deal with the case that for all good lines $l\subset X$ the curve $D_l$ is reducible, and obtain a contradiction by analysing the singularities of $D_l$ (that are in correspondence with that of $X$). We finish the proof by considering the case that for one good line $l$ the curve $D_l$ is irreducible: we adapt an argument of \cite[\S 2]{LSV} to achieve the desired contradiction.

\subsection{Very good lines}\label{v good lines}
We begin by recalling the definitions.

\begin{definition}\cite[Def 3.4]{CML09},\cite[Def 2.2]{LSV}
     A line $l\subset X$ is \textbf{good} if for any plane $P\subset \bP^4$ containing $l$ and another line, the intersection $P\cap X$ consists of three distinct lines.
\end{definition}

Projecting from a line $l\subset X$ gives a conic bundle structure $f:Bl_lX\rightarrow \bP^2$, where the discriminant curve $D_l\subset \bP^2$ has degree $5$. This curve has a natural double cover $\tD\rightarrow D$, where $\tD$ parametrises the lines in $X$ intersecting $l$. 

We denote by $F(X)$ the Fano variety of lines contained in the cubic threefold $X$ - when $X$ is smooth, this is a smooth surface. For $X$ with isolated singularities, $F(X)$ is singular along lines that pass through a singularity. Note that the curve $\tD\subset F(X).$

\begin{proposition}\cite[Prop 3.6]{CML09}\label{prop: good line sings of D}
    Let $l\subset X$ be a good line, $D_l$ and $\tD$ as above. Then:
    \begin{enumerate}
        \item There exists a natural 1-to-1 correspondence between the singularities of $D_l$ and those of $X$, including the analytic type.
        \item The double cover $\tD\rightarrow D_l$ is \'etale.
    \end{enumerate}
\end{proposition}
\begin{remark}
    A cubic threefold with a singularity of corank $3$ does not have a good line - indeed if it did, then the quintic curve $D_l$ would have a singularity of corank 3, which is impossible.
\end{remark}

\begin{definition}\cite[Defn 2.9]{LSV}\label{defn: v good}
    Let $X\subset \bP^4$ be a cubic threefold, and $l\subset X$ a good line. Then $l$ is \textbf{very good} if the curve $\tD$ of lines in $X$ meeting $l$ is irreducible.
\end{definition}

In \cite[\S 3.2]{CML09}, the authors show that a cubic threefold $X$ has a good line if its total Milnor number $\mu_{tot}(X)=\sum_{q\in \Sing(X)}\mu(q)$ is less than or equal to $5$, or if the singularities of $X$ are \textbf{allowable}  (see also \cite[\S 2]{LSV}): recall that a  cubic threefold $X$ has allowable singularities if it has singularities of type $A_k$ with $k\leq 5$ or $D_4$ singularities.

We will first extend this result, and prove existence of good lines for a cubic threefold with $\defect=0$. Notice that there exist cubics with $\defect=0$ but non-allowable singularities. On the other hand, having $\mu_{tot}(X)\leq 5$ or allowable singularities does not imply $\defect=0$. For instance, there exist cubic threefolds containing a plane with exactly four nodes, thus with $\mu_{tot}(X)=4$ and allowable singularities.

We will use the following lemma, which seems to be well-known but we could not locate a reference.

\begin{lemma}\label{lem: F(X) irreducible}
    Let $X$ be a cubic threefold with $\sigma(X)=0,$ and assume there are no singularities of corank 3. Then $F(X)$ is irreducible.
\end{lemma}

\begin{proof}
    Assume first that at least one singular point $q\in X$ is of corank $1$, and let $C_q$ be the corresponding curve of lines through $q$. Since $\sigma(X)=0,$ by Theorem \ref{defect} the curve $C_q$ is irreducible. We claim $F(X)$ is birational to $\Sym^2(C_q)$, and is thus irreducible. 
    Recall that $F(X)$ is singular at $[m]\in F(X)$ if and only if the line $m\subset X$ passes through a singular point of $X$ (\cite[Cor 1.11]{AltKlei} and\cite[Cor 4]{BVdV}).
    Let $l\in F(X)\setminus \Sing(F(X))$, and consider the plane $H:=\langle q, l\rangle\subset \bP^4$. Note that because $\sigma(X)=0,$ this plane is not contained in $X$ (by applying Lemma \ref{plane => not palindromic}), and we intersect $H\cap X$. Since $X$ is singular at $q$, we have $H\cap X=l\cup l_x\cup l_y$, and $l_x, l_y$ both contain $q$. Thus under the projection $\pi:Bl_qX\rightarrow \bP^3$, both $\pi(l_x)=x$ and $\pi(l_y)=y$ are contained in $C_q$. Thus $\phi([l])=(x,y)$ defines a rational map $\phi:F(X)\dashrightarrow \Sym^2(C_q).$ We construct the rational inverse as follows: let $(x,y)\in \Sym^2(C_q)$ where $x\neq y$. Then there exists lines $l_x, l_y$ parameterised by $x,y\in C_q$ passing through the singular point $q;$ we take the residual line of the intersection $\langle l_x, l_y\rangle \cap X.$

    Now suppose that all singular points are corank $2$, and project from such a singular point $q\in X$. In this case, the quadric $Q_q=P_1\cup P_2$ where $P_i\cong \bP^2$. Since $\sigma(X)=0$, by Theorem \ref{defect} the curve $C_q=C_1\cup C_2$ where $C_i\subset P_i$ is an irreducible cubic curve.

    We claim that $F(X)$ is birational to $C_1\times C_2,$ and is thus irreducible. Let $[l]\in F(X)\setminus \Sing(F(X)),$ and consider the plane $H:=\langle q, l\rangle\subset \bP^4$. Again, intersecting $H\cap X=l\cup l_x\cup l_y$ where $l_x, l_y$ are lines in $X$ passing through $q,$ and are thus parameterised by points $x,y\in C_q.$ Note that $\pi(l)$ is a line in $\bP^3$, and thus intersects each $P_i$ in a single point ($\pi(l)$ cannot be contained in $P_i$ - if it was, by analysing the equation of $X$ it must be a component of the $(2,3)$ intersection curve $C_q$). 

    It follows that exactly one of the points $x,y$ belongs to each component $C_i;$ we assume $x\in C_1,$ $y\in C_2.$ Thus $\phi([l])=(x,y)$ defines a rational map $$\phi: F(X)\dashrightarrow C_1\times C_2.$$
    We construct the rational inverse: let $(x,y)\in C_1\times C_2$, with $x\neq y$. Then there exists lines $l_x, l_y$ parameterised by $x,y\in C_1\cup C_2$ passing through the singular point $q$; we take the residual line of the intersection $\langle l_x, l_y\rangle \cap X.$
\end{proof}
\begin{remark}\label{remark: F(X) reducible}
    The Fano variety of lines has been studied in the case of $\sigma(X)>0.$ For example, in \cite{flops} the authors describe the components of $F(X)$ explicitly for a general cubic threefold containing a rational normal cubic scroll (see also \cite{dolgnodal}); in particular there are 3 components. The birational geometry of a Fano variety for a cubic threefold containing a plane appears in \cite[\S 3.3]{MR1738983}, and in detail in \cite{CBrooke} over fields of characteristic not 2. Again, there are three components.
\end{remark}

It follows that the general line on a cubic threefold with $\defect=0$ is a good line. More precisely:

\begin{lemma}\label{good lines}
    Let $X$ be a cubic threefold with $\sigma(X)=0$, and assume there are no singularities of corank $3$. Then the set of good lines $U\subset F(X)$ is a (Zariski) open subset of $F(X).$
\end{lemma}

\begin{proof}
   Since $\sigma(X)=0,$ the Fano variety of lines $F(X)$ is reduced and irreducible by Lemma \ref{lem: F(X) irreducible}. The argument follows as in \cite[Lemma 3.9]{CML09}.
\end{proof}

\subsection{Containing a plane or a cubic scroll}\label{subsec: plane or scroll=> no v good line}
In this subsection, we will show that a cubic threefold containing a plane or a rational normal cubic scroll (and thus by Theorem 
 \ref{defect equiv to plane or scroll} $\defect>0$) does not contain a very good line.
This proves one direction of Theorem \ref{thm: v good iff defect=0}.

\begin{proposition}\label{prop: plane implies no v good line}
    Suppose that $X$ contains a plane or a rational normal cubic scroll. Then for every good line $l\in X$, the corresponding  curve $\widetilde{D}_l$ is reducible. In particular, $X$ does not contain a very good line.
\end{proposition}
\begin{proof}
    Since $X$ contains a plane or a cubic scroll, the Fano variety of lines is reducible (see Remark \ref{remark: F(X) reducible}). Let $F(X):=\cup F_i$, where each $F_i$ is an irreducible surface. 
    First we will show that $\widetilde{D}_l$ is not contained in $\Sing(F(X))$. Recall that by definition,
    $$\widetilde{D}_l:=\{[l']\in F(X)| l'\cap l\neq \varnothing\}.$$ Suppose $\widetilde{D}_l\subset \Sing(F(X));$ then every line $[l']\in \widetilde{D}_l$ passes through a singular point of $X$. Under the projection $f:Bl_lX\rightarrow \bP^2$, the pairs of lines $l', \tau(l')$ (where here $\tau$ is the covering involution associated to $\widetilde{D}_l\rightarrow D_l)$ map to a point on the discriminant curve $D_l$. 
    Since $l$ is a good line, the singularities of $D_l$ are in one-to-one correspondence with those of $X$, hence the image under projection of a singular point $q$ of $X$ is a singular point. Since every line in the fibers parametrised by $D_l$ passes through a singular point, every point of $D_l$ is singular - this is a contradiction to $X$ having isolated singularities.

Assume now that $\widetilde{D}_l$ is irreducible. Since $\tD\not\subset \Sing(X)$, it follows that $\widetilde{D}_l$ is contained in a unique component of $F(X),$ say $F_1$. For $i\neq 1$, let
$$Z_i:=\{x\in X : \exists [l']\in F_i \text{ with } x\in l'\}.$$
 There are two possibilities, either $Z_i$ covers $X,$ or $Z_i$ is a surface contained in $X$. If $Z_i$ covers $X$, then for every point $x\in l$ there exists a line $[l']\in F_i$ intersecting $l$ at $x$ - it follows that $\widetilde{D}_l$ has a component contained in $F_i$, and is thus reducible.

 Now assume that $Z_i\subset X$ is a surface. Let $p:\bL\rightarrow F(X)$ be the universal line over $F(X)$, with projection $q:\bL\rightarrow X$. Note that $p$ is a $\bP^1$-bundle over $F(X)$, and $q$ is generically finite.
 Let $p_i:\bL_i\rightarrow F_i$, $q_i:\bL_i\rightarrow Z_i$ be the restrictions to the component $F_i$.
 Then $\bL_i$ is a $\bP^1$ bundle over the surface $F_i$, and $q_i$ is a dominant map to the surface $Z_i$. In particular, the generic fiber of $q_i$ is $1$-dimensional. 
 Since $Z_i\subset X\subset \bP^4$, the intersection of the good line $l$ and $Z_i$ is non empty; pick $z\in Z_i\cap l$, and consider the fiber $q_i^{-1}(z).$ 
 Since the dimension of the fibers of $q_i$ is an upper-semi continuous function, $\dim q_i^{-1}(z)\geq 1.$ In particular, there is a $1-$parameter family of lines contained in $F_i$ that pass through $z$, parametrised by a curve $C_z\subset F_i$. Since $z\in l$, the curve $C_z
\subset F_i$ is a component of $\widetilde{D}_l$, and $\widetilde{D}_l$ is reducible.
\end{proof}

\subsection{Assuming non-existence of a very good line}\label{no v good line}
In order to complete the proof of Lemma \ref{thm: v good iff defect=0}, it remains to prove the converse: if $X$ does not have a very good line, then $\defect>0$. Note that if $X$ does not have a good line, then $\sigma(X)>0$ by Lemma \ref{good lines}, or $X$ has a corank 3 singularity: we thus assume the existence of a good line in what follows. Further, a cubic threefold has a very good line as long as it appears as a hyperplane section of a general cubic fourfold \cite[Proposition 2.10]{LSV}. In particular, any cubic threefold with $\mu(X) \leq 5$ has a very good line. Thus we can assume that $X$ has $\mu(X)\geq 6$.

Suppose that $l\subset X$ is a good line, but is not very good. Then by definition the associated curve
$\widetilde{D}_l\subset F(X)$ has multiple components, but the discriminant curve $D_l:=\widetilde{D}_l/\tau\subset \bP^2$ can be either reducible or irreducible. Theorem \ref{thm: v good iff defect=0} will then follow from Theorem \ref{claim: quintic reducible}, which considers the case where $D_l$ is reducible, and Theorem \ref{claim:quintic irreducible}, which considers the case where $D_l$ is irreducible.

\begin{theorem}\label{claim: quintic reducible}
    Let $X$ be a cubic fourfold with isolated singularities such that there exists at least one good line (in particular, $X$ has no corank 3 singularity).
    Suppose that for all good lines $l\in X$, the discriminant curve $D_l$ is reducible (in particular, $\mu(X)\geq 6)$. Then $X$ contains a plane or a cubic scroll.
\end{theorem}

In order to prove Theorem \ref{claim: quintic reducible}, there are two situations to consider: 
\begin{enumerate}
        \item for every good line $l$, the curve $D_l$ contains a line component $L_l\subset D_l$; or
        \item for at least one good line the curve $D_l$ is the union of a smooth conic and an irreducible plane cubic.
    \end{enumerate}

\begin{proposition} \label{line component case}
Let $X$ be as in Theorem \ref{claim: quintic reducible}.
    Suppose that for every good line $l\subset X$, the quintic $D_l$ contains a line component $L_l.$ Then $\sigma(X)>0.$
\end{proposition}
\begin{proof}
    Assume that $\sigma(X)=0.$ Then $F(X)$ is irreducible and reduced by Lemma \ref{lem: F(X) irreducible}, and the set of good lines is open in $F(X).$ For each good line $l\subset X$, let $D_l=L_l\cup \Gamma_l$, where $L_l$ is the line component, and $\Gamma_l$ is a (possibly reducible) quartic. 

    Denote by $S_l:=\pi^{-1}(L_l)$ the preimage in $X$ of the line component under the projection map $\pi:Bl_lX\rightarrow \bP^2$. Then $S_l$ is a cubic surface in $X$ that contains a 1-parameter family of (pairs) of lines that meet the given line $l$. If $S_l$ is reducible for some $l$, then $X$ contains a plane. Thus we assume $S_l$ is irreducible for every choice of good line $l$. 

    For every plane $P$ that contains $l$, the intersection $P\cap S_l$ has a residual singular conic. The singular point moves in $S$ -- it follows by Bertini that $\Sing(S_l)=M'_l$ is a line. 
    This line $M'_l\subset S_l\subset X$ necessarily passes through a singular point of $X$; indeed, $L_l$ intersects the quartic $\Gamma_l$ in at least one point, whose preimage is a singular point of $X$ contained in $M'_l.$ It follows that $X$ contains an infinite number of cubic surfaces that are singular along a line passing through a singular point of $X$. This is a contradiction by Lemma \ref{lem: only finite bad surfaces} below.
\end{proof}

\begin{lemma}\label{lem: only finite bad surfaces}
    Let $X$ be a cubic threefold with singularities of corank $\leq 2.$ Assume $\sigma(X)=0.$ Then $X$ contains only finitely many cubic surfaces which are singular along a line.
\end{lemma}
\begin{proof}
This is essentially \cite[Lemma 2.6]{LSV}
 -- their proof only uses the fact that $F(X)$ is irreducible, and thus holds for $X$. 
    The conclusion of the proof of the lemma is that this forces the curve $C_q$ to be nonreduced, which is true only in the corank 3 case. Thus the result holds.
\end{proof}

\begin{proof}[Proof of Theorem \ref{claim: quintic reducible}]
    Assume that $X$ does not contain a plane or a cubic scroll. In particular, we assume $\sigma(X)=0.$
    By Proposition \ref{line component case}, it suffices to consider the situation when there exists at least one good line $l\subset X$ such that $D_l$ is the union of an irreducible cubic $T_l$ and a smooth conic $Q_l$. We will analyse the singularities of $D_l$ and use Theorem \ref{defect} to conclude that the defect of $X$ should be positive, leading to a contradiction.

    Since the component $T_l$ of $D_l$ is irreducible, it is either smooth or has one singularity. If $T_l$ is smooth, then $D_l$ has only $A_n$ singularities occuring as the six intersection points of $T_l$ and $Q_l$ (counted with multiplicities). The combination of $A_n$ singularities is determined by a partition of six into a sum: there are eleven possibilities. In the case of $6$  $A_1$ singularities (i.e. the partition 1+1+1+1+1+1), $X$ contains a cubic scroll or a plane, since the singularities of $X$ are the same as those of $D_l$ by Proposition \ref{prop: good line sings of D}. In the remaining cases, the proofs of \cite[Propositions 4.6-4.9]{viktorova2023classification}) imply that the curve $C_q$ associated to any singularity $q$ of $X$ is reducible. By Theorem \ref{defect}, it follows that $\sigma(X)>0,$ a contradiction.
    
    Next we assume that $T_l$ has a node. If the smooth conic $Q_l$ does not pass through the node, then $X$ has a combination of $A_n$ singularities again, and we proceed as above. If $Q_l$ passes through the node, then the corresponding singularity of $D_l$ has type either $D_4, D_6, D_8, D_{10}$ or $D_{12}$. Singularities of types $D_{n}$ with $n>8$ do not occur on cubic threefolds by \cite[Theorem I]{viktorova2023classification}, and thus $D_{10}$ or $D_{12}$ do not occur on $D_l$. Analysing the possible combinations of singularities on $X$ following \cite[Section 3.2]{viktorova2023classification}, we see immediately that $\defect>0$ in all the cases in question.
    
    Finally, we can assume that $T_l$ has a cusp and $Q_l$ passes through the cusp. If the resulting singularity is $ADE$, then it is of $D_5$, $D_7$ or $E_7$ type. The case by case analysis again follows from \cite[Section 3.2]{viktorova2023classification} and leads to the conclusion that $\sigma(X)>0$. If the resulting singularity of $D_l$ is non-$ADE$ of corank 2, then projection from the corresponding singularity $q\in X$ results in a reducible curve $C_q$ with at least three components, also implying $\sigma(X)>0$. Corank 3 singularities are impossible since we are considering a plane curve.
\end{proof}
 Finally, we finish the proof of Theorem \ref{thm: v good iff defect=0} by considering the case when the quintic plane curve $D_l$ is irreducible.
\begin{theorem}\label{claim:quintic irreducible}
   Let $X$ be a cubic threefold with isolated singularities of corank $\leq 2.$ Suppose that for every good line $l\in X$ the curve $\widetilde{D}_l$ is reducible. Then $\sigma(X)>0$.
\end{theorem}

\begin{proof}
    It suffices to assume that $D_l\subset\bP^2$ is an irreducible quintic plane curve for some choice of good line $l\subset X$. Indeed, if no such good line exists, then $D_l$ is reducible for all good lines and we are in the set up of Theorem \ref{claim: quintic reducible}. 

    Let $l\subset X$ be such a good line as above. Since $\widetilde{D}_l\rightarrow D_l$ is an \'etale double cover, $\widetilde{D}_l=\widetilde{D}_l^1\cup \widetilde{D}_l^2$. Assume for a contradiction that $\sigma(X)=0.$ This implies $F(X)$ is irreducible and reduced by Lemma \ref{lem: F(X) irreducible}. We follow the strategy of \cite[\S 2]{LSV}: almost all of the arguments work with the assumption that $F(X)$ is irreducible and reduced.
    
    Consider the incidence variety $P\subset F(X)\times X$, which is a $\bP^1$ bundle $p:P\rightarrow F(X).$
    Since $F(X)$ is irreducible and reduced, so is $P$. 
    Note also that the degree of $q:P\rightarrow X$ is $6$ as for a smooth cubic threefold, and the degree of $pr_2:(P\times_X P)\setminus \Delta_P\rightarrow P$ is $5$. We first claim that the assumption of $\tD$ being reducible implies that  $(P\times_X P)\setminus \Delta_X$ has two irreducible components dominating $X$.
   Indeed, this is \cite[Lemma 2.12]{LSV}; the two components are given as $$\calC_1=\bigcup_{[l]\in F(X)}\tD^1 \text{ and }\calC_2=\bigcup_{[l]\in F(X)}\tD^2.$$

Next, we will show that $(P\times_XP)\setminus\Delta_P$ is irreducible, giving a contradiction as in \cite[\S 2]{LSV}.
First, we note that every component $Z$ of $(P\times_XP)\setminus\Delta_P$ dominates $P$ by the second projection map $pr_2$. In \cite[Lemma 2.14]{LSV}, the authors claim that the only possibility for $Z$ not to dominate $P$ is if there is a curve $W\subset X$ such that for every $x\in W,$ there is a curve $C_x$ of lines passing through $x$. This is impossible by \cite[Claim 2.15]{LSV}: there are only finitely many such points $x\in X.$

Recall that we have the map
\[\begin{tikzcd}
    P\times_XP\arrow[r, "pr_2"]& P\arrow[r, "p"]& F(X)
\end{tikzcd}\]
and the fiber over $[l]\in F(X)$ is identified with $\tD.$ Consider the involution $i$ of $P\times_XP\setminus \Delta_P$ mapping $(l_1, l_2,x)$ where $l_1\cap l_2=x$ to $(l_3, l_2, z)$ where $l_3$ is the residual line of $\langle l_1, l_2\rangle \cap X,$ and $z=l_3\cap l_2.$ When restricted to $\tD,$ this is the involution $\tau$ giving the double cover $\tD\rightarrow D_l.$ 
By the proof of \cite[Lemma 2.16]{LSV}, since $D_l=\tD/\tau=\tD/i$ is irreducible, then so is $(P\times_XP)\setminus \Delta_P/i$. It follows that $(P\times_XP)\setminus \Delta_P$ has at most two components that are exchanged under the involution.

We have shown that ($P\times_XP)\setminus \Delta_P=\calC_1\cup\calC_2$, where $\calC_1=\bigcup_{[l]\in F(X)}\tD^1$ and $\calC_2=\bigcup_{[l]\in F(X)}\tD^2$. Let $k_1, k_2$ be the degrees of $pr_2:\calC_i\rightarrow P$; we have either $(k_1,k_2)=(2,3), (1,4)$ since these are the only possible degrees of the components $(\tD^1, \tD^2)$ over $l$. To reach the contradiction, one applies the same arguments as in the proof of \cite[Lemma 2.13]{LSV}, which culminates in a contradiction of the possibility of the degrees $(k_1,k_2)$ -- we omit the details.

Thus assuming $\tD$ is reducible for all good lines $l$, we have shown that if $\sigma(X)=0$ then $P\times_X P\setminus \Delta_P$ is both reducible and irreducible - a contradiction. Thus $\sigma(X)>0.$

\end{proof}

\bibliographystyle{alpha}
\bibliography{bibliography}

\begin{thebibliography}{CMGHL21}

\bibitem[AGZV85]{Arnold2012}
V.I. Arnold, S.M. Gusein-Zade, and A.N. Varchenko.
\newblock {\em Singularities of Differentiable Maps}, volume~1.
\newblock Birkh\"{a}user, 1 edition, 1985.

\bibitem[AK77]{AltKlei}
Allen~B. Altman and Steven~L. Kleiman.
\newblock Foundations of the theory of {F}ano schemes.
\newblock {\em Compositio Math.}, 34(1):3--47, 1977.

\bibitem[All03]{Allcock}
Daniel Allcock.
\newblock The moduli space of cubic threefolds.
\newblock {\em J. Algebraic Geom.}, 12(2):201--223, 2003.

\bibitem[Aur90]{MR1096457}
Alf~Bj{\o}rn Aure.
\newblock The smooth surfaces on cubic hypersurface in {${\bf P}^4$} with
  isolated singularities.
\newblock {\em Math. Scand.}, 67(2):215--222, 1990.

\bibitem[Avi18]{Avilov}
Artem Avilov.
\newblock Automorphisms of singular three-dimensional cubic hypersurfaces.
\newblock {\em Eur. J. Math.}, 4(3):761--777, 2018.

\bibitem[Bea77]{MR0572974}
Arnaud Beauville.
\newblock Prym varieties and the {S}chottky problem.
\newblock {\em Invent. Math.}, 41(2):149--196, 1977.

\bibitem[BGM25]{BGM25}
Simone Billi, Annalisa Grossi, and Lisa Marquand.
\newblock Cubic fourfolds with a symplectic automorphism of prime order, 2025.

\bibitem[Bro18]{Brosnan}
Patrick Brosnan.
\newblock Perverse obstructions to flat regular compactifications.
\newblock {\em Mathematische Zeitschrift}, 290(1):103--110, 2018.

\bibitem[Bro23]{CBrooke}
Corey Brooke.
\newblock Lines on {C}ubic {T}hreefolds and {F}ourfolds {C}ontaining a {P}lane.
\newblock {\em Ph.D. thesis, University of Oregon}, 2023.

\bibitem[BVdV79]{BVdV}
W.~Barth and A.~Van~de Ven.
\newblock Fano varieties of lines on hypersurfaces.
\newblock {\em Arch. Math. (Basel)}, 31(1):96--104, 1978/79.

\bibitem[Che10]{CheltsovFactorial3fold}
Ivan Cheltsov.
\newblock Factorial threefold hypersurfaces.
\newblock {\em J. Algebraic Geom.}, 19(4):781--791, 2010.

\bibitem[Cle83]{Clemens}
C.~Herbert Clemens.
\newblock Double solids.
\newblock {\em Adv. in Math.}, 47(2):107--230, 1983.

\bibitem[CMGHL21]{MR4218962}
Sebastian Casalaina-Martin, Samuel Grushevsky, Klaus Hulek, and Radu Laza.
\newblock Complete moduli of cubic threefolds and their intermediate
  {J}acobians.
\newblock {\em Proc. Lond. Math. Soc. (3)}, 122(2):259--316, 2021.

\bibitem[CML09]{CML09}
Sebastian Casalaina-Martin and Radu Laza.
\newblock The moduli space of cubic threefolds via degenerations of the
  intermediate {J}acobian.
\newblock {\em J. Reine Angew. Math.}, 633:29--65, 2009.

\bibitem[CMTZ25]{CMTZ}
Ivan Cheltsov, Lisa Marquand, Yuri Tschinkel, and Zhijia Zhang.
\newblock Equivariant geometry of singular cubic threefolds, {II}.
\newblock {\em J. Lond. Math. Soc.}, 2025.
\newblock To appear.

\bibitem[CP10]{CheltsovSextic}
Ivan Cheltsov and Jihun Park.
\newblock Sextic double solids.
\newblock In {\em Cohomological and geometric approaches to rationality
  problems}, volume 282 of {\em Progr. Math.}, pages 75--132. Birkh\"{a}user
  Boston, Boston, MA, 2010.

\bibitem[CPS19]{MR3912057}
Ivan Cheltsov, Victor Przyjalkowski, and Constantin Shramov.
\newblock Which quartic double solids are rational?
\newblock {\em J. Algebraic Geom.}, 28(2):201--243, 2019.

\bibitem[Cut88]{Cut}
Steven Cutkosky.
\newblock Elementary contractions of {G}orenstein threefolds.
\newblock {\em Math. Ann.}, 280(3):521--525, 1988.

\bibitem[Cyn01]{Cyn}
S{\l}awomir Cynk.
\newblock Defect of a nodal hypersurface.
\newblock {\em Manuscripta Math.}, 104(3):325--331, 2001.

\bibitem[Dim86]{Dim86}
Alexandru Dimca.
\newblock On the homology and cohomology of complete intersections with
  isolated singularities.
\newblock {\em Compositio Mathematica}, 58(3):321--339, 1986.

\bibitem[Dim90]{Dimca}
Alexandru Dimca.
\newblock Betti numbers of hypersurfaces and defects of linear systems.
\newblock {\em Duke Math. J.}, 60(1):285--298, 1990.

\bibitem[DM25]{DMupcoming}
Yajnaseni Dutta and Lisa Marquand.
\newblock Relative compactified prym and picard fibrations associated to very
  good cubic fourfolds.’’.
\newblock {\em Proc. Amer. Math. Soc.}, 2025.
\newblock To appear.

\bibitem[Dol16]{dolgnodal}
Igor Dolgachev.
\newblock Corrado {S}egre and nodal cubic threefolds.
\newblock In {\em From classical to modern algebraic geometry}, Trends Hist.
  Sci., pages 429--450. Birkh\"{a}user/Springer, Cham, 2016.

\bibitem[dPW00]{DPW}
A.~A. du~Plessis and C.~T.~C. Wall.
\newblock Singular hypersurfaces, versality, and {G}orenstein algebras.
\newblock {\em J. Algebraic Geom.}, 9(2):309--322, 2000.

\bibitem[FW89]{Finkelnberg}
Hans Finkelnberg and J\"{u}rgen Werner.
\newblock Small resolutions of nodal cubic threefolds.
\newblock {\em Nederl. Akad. Wetensch. Indag. Math.}, 51(2):185--198, 1989.

\bibitem[Had00]{MR1738983}
Ingo Hadan.
\newblock Tangent conics at quartic surfaces and conics in quartic double
  solids.
\newblock {\em Math. Nachr.}, 210:127--162, 2000.

\bibitem[Has00]{hassett}
Brendan Hassett.
\newblock Special cubic fourfolds.
\newblock {\em Compositio Math.}, 120(1):1--23, 2000.

\bibitem[Hav16]{Havasi}
Kriszti\'{a}n Havasi.
\newblock Geometric realization of strata in the boundary of the intermediate
  jacobian locus.
\newblock {\em Ph.D. thesis, University of Colorado at Boulder}, 2016.

\bibitem[HT10]{flops}
Brendan Hassett and Yuri Tschinkel.
\newblock Flops on holomorphic symplectic fourfolds and determinantal cubic
  hypersurfaces.
\newblock {\em J. Inst. Math. Jussieu}, 9(1):125--153, 2010.

\bibitem[Kal11]{MR2729277}
Anne-Sophie Kaloghiros.
\newblock The defect of {F}ano 3-folds.
\newblock {\em J. Algebraic Geom.}, 20(1):127--149, 2011.

\bibitem[KL25]{KL20}
Matt Kerr and Radu Laza.
\newblock Hodge theory of degenerations, ({II}): vanishing cohomology and
  geometric applications.
\newblock {\em "Current developments in Hodge theory", Proceedings of Hodge
  theory at IMSA, Simons Symposia (Springer)}, 2025.

\bibitem[KLS21]{KLS19}
Matt Kerr, Radu Laza, and Morihiko Saito.
\newblock Hodge theory of degenerations, ({I}): consequences of the
  decomposition theorem.
\newblock {\em Selecta Math. (N.S.)}, 27(4):Paper No. 71, 48, 2021.

\bibitem[Koe92]{MR1191602}
Laurent Koelblen.
\newblock Surfaces de {${\bf P}_4$} trac\'{e}es sur une hypersurface cubique.
\newblock {\em J. Reine Angew. Math.}, 433:113--141, 1992.

\bibitem[LSV17]{LSV}
Radu Laza, Giulia Sacc\`a, and Claire Voisin.
\newblock A hyper-{K}\"{a}hler compactification of the intermediate {J}acobian
  fibration associated with a cubic 4-fold.
\newblock {\em Acta Math.}, 218(1):55--135, 2017.

\bibitem[Mar23]{MR4531678}
Lisa Marquand.
\newblock Cubic fourfolds with an involution.
\newblock {\em Trans. Amer. Math. Soc.}, 376(2):1373--1406, 2023.

\bibitem[MM25a]{marquand2023classification}
Lisa Marquand and Stevell Muller.
\newblock Classification of symplectic birational involutions of manifolds of
  {O}{G}10 type.
\newblock {\em Mathematische Zeitschrift}, 309(4):65, 2025.

\bibitem[MM25b]{marquand2023finite}
Lisa Marquand and Stevell Muller.
\newblock Finite groups of symplectic birational transformations of {I}{H}{S}
  manifolds of ${O}{G}10$ type.
\newblock {\em Forum. Math. Sigma}, 2025.
\newblock To appear.

\bibitem[MO22]{MR4395101}
Giovanni Mongardi and Claudio Onorati.
\newblock Birational geometry of irreducible holomorphic symplectic tenfolds of
  {O}'{G}rady type.
\newblock {\em Math. Z.}, 300(4):3497--3526, 2022.

\bibitem[NS95]{NS95}
Yoshinori Namikawa and J.~H.~M. Steenbrink.
\newblock Global smoothing of {C}alabi-{Y}au threefolds.
\newblock {\em Invent. Math.}, 122(2):403--419, 1995.

\bibitem[PS08]{mixedhodge}
Chris A.~M. Peters and Joseph H.~M. Steenbrink.
\newblock {\em Mixed {H}odge structures}, volume~52 of {\em Ergebnisse der
  Mathematik und ihrer Grenzgebiete. 3. Folge. A Series of Modern Surveys in
  Mathematics}.
\newblock Springer-Verlag, Berlin, 2008.

\bibitem[Sac23]{sac2021birational}
Giulia Sacc{\`a}.
\newblock Birational geometry of the intermediate jacobian fibration of a cubic
  fourfold.
\newblock {\em Geom. Topol.}, 27(4):1479--1538, 2023.

\bibitem[Shr08]{MR2475055}
K.~A. Shramov.
\newblock On the birational rigidity and {$\Bbb Q$}-factoriality of a singular
  double covering of a quadric with branching over a divisor of degree 4.
\newblock {\em Mat. Zametki}, 84(2), 2008.

\bibitem[Ste83]{Ste83}
J.~H.~M. Steenbrink.
\newblock Mixed {H}odge structures associated with isolated singularities.
\newblock In {\em Singularities, {P}art 2 ({A}rcata, {C}alif., 1981)},
  volume~40 of {\em Proc. Sympos. Pure Math.}, pages 513--536. Amer. Math.
  Soc., Providence, RI, 1983.

\bibitem[Ste97]{duBoisInv}
Joseph H.~M. Steenbrink.
\newblock Du {B}ois invariants of isolated complete intersection singularities.
\newblock {\em Ann. Inst. Fourier (Grenoble)}, 47(5):1367--1377, 1997.

\bibitem[Vik25]{viktorova2023classification}
Sasha Viktorova.
\newblock On the classification of singular cubic threefolds.
\newblock {\em Trans. Amer. Math. Soc.}, 2025.
\newblock To appear.

\bibitem[Voi07]{Voi1}
Claire Voisin.
\newblock {\em Hodge theory and complex algebraic geometry. {I}}, volume~76 of
  {\em Cambridge Studies in Advanced Mathematics}.
\newblock Cambridge University Press, Cambridge, {E}nglish edition, 2007.
\newblock Translated from the French by Leila Schneps.

\bibitem[Yok02]{Yokoyama}
Mutsumi Yokoyama.
\newblock Stability of cubic 3-folds.
\newblock {\em Tokyo J. Math.}, 25(1):85--105, 2002.

\end{thebibliography}

\end{document}